\documentclass[a4paper,abstracton]{scrartcl}
\usepackage{amssymb}
\usepackage{amsmath}
\usepackage{amsthm}
\usepackage{graphicx}
\usepackage{fancyhdr}              
\usepackage{lastpage}              
                                   
\usepackage[shortlabels]{enumitem}
\usepackage{esint}

\usepackage [english]{babel}

\usepackage{setspace}

\newcommand{\setR}{\mathbb{R}}
\newcommand{\setC}{\mathbb{C}}
\newcommand{\setZ}{\mathbb{Z}}
\newcommand{\setQ}{\mathbb{Q}}
\newcommand{\setN}{\mathbb{N}}
\newcommand{\setP}{\mathbb{P}}

\newtheorem{definition}{Definition}
\newtheorem{theorem}[definition]{Theorem}
\newtheorem{lemma}[definition]{Lemma}

\newcommand{\divides}{\mid}

\DeclareMathOperator{\Tr}{Tr}

\renewcommand{\pmod}[1]{(\text{mod } #1)}

\begin{document}
\title{Power-Free Values of Polynomials}
\author{Thomas Reuss\\ 
		Mathematical Institute, University of Oxford\\
		reuss@maths.ox.ac.uk}
\maketitle

\begin{abstract}
Let $f$ be an irreducible polynomial of degree $d\geq 3$ with no fixed prime divisor. We derive an asymptotic formula for the number of primes $p\leq x$ such that $f(p)$ is $(d-1)$-free.
\end{abstract}

\renewcommand{\abstractname}{Acknowledgment}
\begin{abstract}
I am very grateful to my supervisor Roger Heath-Brown for many valuable discussions and helpful comments on this paper.\\
I am also very grateful to the EPSRC\footnote{DTG reference number: EP/J500495/1} and to St. Anne's College, Oxford who are generously funding and supporting this project.
\end{abstract}

\section{Introduction}
Let $k$ and $n$ be integers such that $k\geq 2$. Then $n$ is said to be $k$-free if there is no prime
$p$ such that $p^k\divides n$. For an irreducible polynomial $f(x)\in\setZ[x]$ of degree $d$, one expects in general that the set $f(\setZ)=\{f(n),n\in\setZ\}$ contains
infinitely many $k$-free values. This is clearly not true if $f$ has a fixed $k$-th power prime divisor, that is, if there exists a prime $p$ such that $p^k\divides f(n)$ for all $n\in\setZ$. One can conjecture that this is
the only condition under which $f(\setZ)$ fails to contain infinitely many
$k$-free values. In 1933, Ricci \cite{Ricci} proved this conjecture for 
$k\geq d$. In fact, he derived an asymptotic formula for the quantity
\begin{equation}\label{eq:eqquaninbeg}
	\#\{n\leq X: f(n)\text{ is } k\text{-free}\}.
\end{equation}
Further progress was made by Erd\H{o}s \cite{Erdos} who proved the conjecture in the case $k=d-1$ for $d\geq 3$. Later, Hooley \cite{Hooley} derived an
asymptotic formula for each such $k$.\bigskip\\
In \cite{Erdos}, Erd\H{o}s proposed the similar question, whether 
$f(\setP)=\{f(p),p\text{ prime}\}$ contains infinitely many $(d-1)$-free values. Hooley \cite{Hooley2} proved this conjecture for $d\geq 51$. Nair
\cite{Nair,Nair2} further refined this result and 
proved Erd\H{o}s' conjecture for $d\geq 7$. Recently, Helfgott 
\cite{Helfgott1,Helfgott2} has established the conjecture for 
$d=3$ and for all quartic polynomials with sufficiently high entropy.
Finally, Browning \cite{Browning} has settled the conjecture for $d\geq 5$.
Thus, the conjecture remains open for irreducible quartic polynomials with $Gal(f)=A_4$ or $Gal(f)=S_4$.\bigskip\\
In this work, we will settle the conjecture for the remaining cases and establish the following theorem:
\begin{theorem}\label{thm:thm1}
	Let $f(x)\in\setZ[x]$ be an irreducible polynomial of degree $d\geq 3$
	and assume that $f$ has no fixed $(d-1)$-th power prime divisor. Define
	\[
		N'_f(X)=\#\{p\leq X: p\text{ prime}, f(p)\text{ is } 
		(d-1)\text{-free}\}.
	\]
	Then, for any $C>1$, we have
	\[
		N'_f(X)=c'_f\pi(X)+O_{C,f}\left(\frac{X}{(\log X)^C}\right),
	\]
	as $x\rightarrow\infty$, where
	\[
		c'_f=\prod_p\left(1-\frac{\rho'(p^{d-1})}{\phi(p^{d-1})}\right),
	\]
	and
	\[
		\rho'(d)=\#\{n\pmod d: (d,n)=1, d\divides f(n)\}.
	\]
\end{theorem}
It should be noted that our methods are sufficiently robust to save an arbitrary power of $\log X$ in the error term which gives us an improvement over Helfgott's results. Indeed, the worst part of our error term comes from the Siegel-Walfisz Theorem and our methods can save a power of $X$ in the error term for the asymptotic formula for the quantity \eqref{eq:eqquaninbeg} when $k=d-1$. More precisely, we get the following theorem for a polynomial $f$ as
in Theorem \ref{thm:thm1}:
\begin{theorem}\label{thm:thm2}
Let
	\[
		N_f(X)=\#\{n\leq X: f(n)\text{ is } 
		(d-1)\text{-free}\}.
	\]
	Then, for some $\delta=\delta(d)$, we have
	\[
		N_f(X)=c_f X+O_{f}(X^{1-\delta}),
	\]
	as $x\rightarrow\infty$, where
	\[
		c_f=\prod_p\left(1-\frac{\rho(p^{d-1})}{p^{d-1}}\right),
	\]
	and
	\[
		\rho(d)=\#\{n\pmod d: d\divides f(n)\}.
	\]
\end{theorem}
The work of Browning \cite{Browning} is in parts a refinement of Heath-Brown \cite{RHBratptsonalgvar}. The key idea is to reduce the problem of Theorem \ref{thm:thm1} to a counting problem where one wants to find an upper bound for the number of points $(p,q,h)$ on the
algebraic variety defined by $f(p)=q^{d-1}h$, where $p$, $q$ and $h$ are restricted to certain sizes. Browning's argument then partially relies on work by Salberger \cite{Salberger} about the density of
integer points on affine surfaces.\bigskip\\
The proofs of Browning and Heath-Brown use the determinant method for which the interested reader may consult
Heath-Brown \cite{cime}. It should be noted that Heath-Brown has applied the approximate determinant method to problems involving power-free values of polynomials previously. In \cite{RHBpowerfree}, he considered irreducible polynomials of the shape $f(x)=x^d+c$. This problem gets then converted into the approximate Diophantine equation $a^kb=f(n)=n^d+O(1)$. And points $(n,a,b)$ therefore lie
close to the weighted projective curve $X_0^d=X_1X_2^k$. Thus, the particular shape of $f$ allows Heath-Brown to consider points close to a curve rather than points on a surface. And since the determinant method seems to be more efficient in counting points on varieties with lower dimension, this provides the key saving in his argument. Heath-Brown is able to handle Theorem \ref{thm:thm1} for $k=d-1$ and $d\geq 3$, provided $f$ has special shape.
\section{Preliminaries}
We will now start the proof of Theorem \ref{thm:thm1} and Theorem \ref{thm:thm2}. We will use the following terminology in this paper.
\begin{itemize}
	\item Pick $\theta\in\setC$ such that $f(\theta)=0$. Then
	$K=\setQ(\theta)$ is a number field and $\mathcal{O}_K$ 
	is the ring of integers of $K$.
	\item Let $\mathcal{B}=\{b_0,\ldots,b_{d-1}\}$ be an integral
	basis of $K$.
	\item Let $\mathcal{B}_\theta$ be the $\setQ$-basis 
	$\{1,\theta,\ldots,\theta^{d-1}\}$ of $K$.
	\item Let $\mathcal{O}_K^\times$ denote the units in $\mathcal{O}_K$.
	\item Let $\Delta_K^2$ be the discriminant of $K$ and let 
	let $\Delta^2(\theta)$ be the discriminant associated to 
	$\mathcal{B}_\theta$.
	\item Let $\langle\alpha_1,\ldots,\alpha_r\rangle$ denote the ideal of
	$\mathcal{O}_K$ generated by $\alpha_1,\ldots,\alpha_r\in\mathcal{O}_K$.
	\item If $\gamma\in K$ then we denote the conjugate of $\gamma$
	under an embedding $\sigma$ by $\gamma^\sigma$.
	\item We write $x\sim X$ to say that $X<x\leq 2X$ and we write
	$x\asymp X$ to say that there exist positive constants $A$, $B$,
	independent of $X$, such that $AX\leq |x|\leq BX$.
\end{itemize}

Our first task is to turn the problem into a problem where we count solutions of a Diophantine equation. More precisely, we shall prove the following lemma:
\begin{lemma}\label{lem:lem1}
Let $X$ be sufficiently large and let $\delta>0$ and $\eta>0$ be arbitrary. Then there exists values $A$ and $B$ with
\[
	X^{1-\delta}\ll A,B\ll X^{1+\delta},
\]
and $\delta'=\delta'(d)>0$ such that
\begin{equation}\label{eq:eqforn2xnx}
	N'_f(2X)-N'_f(X)=c'_f\pi(X)+O\left(\frac{X}{(\log X)^C}\right)
		+O(X^\eta \mathcal{N}(X;A,B)),
\end{equation}
and
\begin{equation}\label{eq:eqforn2xnx2}
	N_f(2X)-N_f(X)=c_f X+O(X^{1-\delta'})
		+O(X^\eta \mathcal{N}(X;A,B)),
\end{equation}
where
\[
	\mathcal{N}(X;A,B)=\#\{(n,a,b)\in\setN^3: n\sim X,
	 a\sim A, b\sim B, \mu^2(a)=1, a^{d-1}b=f(n) \}.
\]
\end{lemma}
\begin{proof}
First, we observe that
\[
	\sum_{m^{d-1}\divides f(p)}\mu(m)
	=\begin{cases}
		1 &\text{if } f(p) \text{ is } (d-1) \text{-free},\\
		0 &\text{otherwise}.
	\end{cases}
\]
Thus, 
\begin{align}
	N'_f(2X)-N'_f(X) &=\sum_{p\sim X} \sum_{m^{d-1}\divides f(p)}\mu(m)
	\nonumber\\
	&=\sum_{m}\mu(m)N'_0(X;m)\label{eq:eqSw},
\end{align}
where 
\[
	N'_0(X;m)=\#\{X<p\leq 2X: f(p)\equiv 0\ \pmod{m^{d-1}}\}.
\]
Recall that
\[
	\rho(d)=\#\{n\pmod d: d\divides f(n)\}.
\]
Hence, $\rho$ is multiplicative and $\rho(p^{d-1})\ll_f 1$ and thus,
\[
	\rho'(m^{d-1})\leq \rho(m^{d-1})\ll_{\epsilon,f} m^\epsilon.
\]
The terms of the sum in \eqref{eq:eqSw} corresponding to small $m\leq (\log X)^{2C}$ will contribute the main term of the asymptotic formula. 
Define $\rho_m=\rho(m^{d-1})$ and let $a_1,\ldots,a_{\rho_m}$ be the solutions of $f(a_i)\equiv 0\ \pmod{m^{d-1}}$. Then, an application of the Siegel-Walfisz Theorem yields a constant $c>0$ such that
\begin{align*}
	N'_0(X;m) &= \sum_{i=1}^{\rho_m} 
			   \#\{X<p\leq 2X: p\equiv a_i\ \pmod{m^{d-1}}\}\\
			&= \sum_{\substack{i\leq\rho_m\\ (a_i,m)=1}}
			   \pi(X;m^{d-1},a_i) + O(\rho(m^{d-1}))\\
			&= \frac{\pi(X)\rho'(m^{d-1})}{\phi(m^{d-1})}+
			   O\left(\frac{\rho(m^{d-1}) X}{\exp(c\sqrt{\log X)}}\right).
\end{align*}
Because of the estimate $\phi(n)\gg_\epsilon n^{1-\epsilon}$, we can conclude that 
\begin{align*}
	\sum_{m\leq (\log X)^{2C}}\frac{\mu(m)\rho'(m^{d-1})}{\phi(m^{d-1})}
		&=\sum_{m=1}^\infty\frac{\mu(m)\rho'(m^{d-1})}{\phi(m^{d-1})}
		+O\left(\sum_{m>(\log X)^{2C}} m^{\epsilon-(d-1)}\right)\\
		&=c'_f+O((\log X)^{-C}).
\end{align*}
Furthermore,
\[
	\sum_{m\leq (\log X)^{2C}}\rho(m^{d-1})\ll
	\sum_{m\leq (\log X)^{2C}} m^\epsilon \ll (\log X)^{4C},
\]
and hence altogether:
\[
	\sum_{m\leq (\log X)^{2C}}\mu(m)N'_0(X;m)
	=c'_f\pi(X)+O\left(\frac{X}{(\log X)^C}\right).
\]
Next, let us consider the contribution to \eqref{eq:eqSw} of the $m$ in
the range $(\log X)^{2C}<m\leq X^{1-\delta}$. For these $m$, we shall
employ the trivial estimate
\begin{align*}
	N'_0(X;m) &\leq\#\{X<n\leq 2X: m^{d-1}\divides f(n)\}
			=\rho(m^{d-1})\left(\frac{X}{m^{d-1}}+O(1)\right)\\
			&\ll Xm^{\epsilon-d+1}+m^\epsilon.
\end{align*}
Thus,
\begin{align*}
	\sum_{(\log X)^{2C}<m\leq X^{1-\delta}}\mu(m)N'_0(X;m)
	&\ll X \sum_{m>(\log X)^{2C}} m^{1-d+\epsilon}+
	\sum_{m\leq X^{1-\delta}}m^\epsilon\\
	&=O\left(\frac{X}{(\log X)^C}\right).
\end{align*}
It remains to find an upper bound for the terms in \eqref{eq:eqSw}
corresponding to the large values of $m>X^{1-\delta}$. We get
\[
	\sum_{m> X^{1-\delta}}\mu(m)N'_0(X;m)
	\ll\#\{(n,a,b)\in\setN^3:
	n\sim X, a>X^{1-\delta}, \mu^2(a)=1, a^{d-1}b=f(n)\}.
\]
After a dyadic subdivision of the ranges for $a$ and $b$, we can deduce that for any $\delta>0$, there exist values $A\gg X^{1-\delta}$ and $B$ such that
the equation \eqref{eq:eqforn2xnx} holds.\smallskip\\ 
One can similarly establish \eqref{eq:eqforn2xnx2}. One starts with the expression
\[
	\sum_{m}\mu(m)\#\{X<n\leq 2X: m^{d-1}\divides f(n)\}.
\]
The terms corresponding to small $m\leq Y$, say contribute
$c_fX+O(X^\epsilon(XY^{2-d}+Y))$ and to optimize the error term, we pick $Y=X^{1/(d-1)}$. The terms with $m$ in the range $X^{1/(d-1)}<m\leq X^{1-\delta}$
contribute $O(X^{(1-\delta)(1+\epsilon)})$. Similarly to the above argument,
the terms with $m>X^{1-\delta}$ contribute the remaining error term
of \eqref{eq:eqforn2xnx2}.\smallskip\\
Our overall goal is therefore to show that $\mathcal{N}(X;A,B)\ll_\epsilon X^{1-\epsilon}$ for some $\epsilon>0$, which will then prove Theorem \ref{thm:thm1} and Theorem \ref{thm:thm2}. Next, we aim to further restrict the ranges of $a$ and $b$. Note that 
\[
	B\ll X^dA^{-(d-1)}\ll X^{1+(d-1)\delta},
\]
since $A\gg X^{1-\delta}$. For our next auxiliary bound, we shall consider the estimate
\begin{equation*}
	\mathcal{N}(X;A,B)\ll 
	\sum_{b\sim B}\#\{(n,a): n\sim X, a\sim A,\mu^2(a)=1, f(n)=a^{d-1}b\}.
\end{equation*}
In the following argument, we fix $b$ and let $n_1,\ldots,n_\nu$ be the solutions of $f(n)\equiv 0\ \pmod{b}$. In particular, $\nu=\rho(b)\ll X^\epsilon$. Thus, we get the estimate
\begin{equation}\label{eq:eqtrivial0}
	\mathcal{N}(X;A,B)\ll 
	\sum_{b\sim B}\sum_{i=1}^\nu\#\{(t,a): a\sim A, t\ll X/B+1,
	\mu^2(a)=1, f(n_i+tb)=a^{d-1}b\}.
\end{equation}
Thus, we shall now count solutions $(t,a)$ of the equation
\begin{equation}\label{eq:eqtrivial1}
	f(n_i+tb)=a^{d-1}b.
\end{equation}
The equation \eqref{eq:eqtrivial1} is of the form $p_1(t)=p_2(a)$, where $p_1$ is a polynomial of degree $d$ and $p_2$ is a polynomial of degree $d-1$. To get an upper bound on the number of pairs $(t,q)$ that satisfy \eqref{eq:eqtrivial1}, we shall employ Heath-Brown \cite[Theorem 15]{RHBratptsonalgvar}. It is easy to see that the polynomial \eqref{eq:eqtrivial1} is absolutely irreducible and with the same notation as in \cite{RHBratptsonalgvar}, we may apply the Theorem with $n=2$, $B_1\asymp X/B+1$ and $B_2\asymp A$, so that $T=\max\{B_1^d,B_2^{d-1}\}\geq B_2^{d-1}$ and hence the points $(t,a)$ satisfying \eqref{eq:eqtrivial1} lie on at most 
$k\ll X^\epsilon B_1^{\frac{1}{d-1}}$. auxiliary curves. Thus, by B\'{e}zout's Theorem and \eqref{eq:eqtrivial0}, we get the estimate
\[
	\mathcal{N}(X;A,B)\ll X^{\delta_1} B(X/B+1)^{\frac{1}{d-1}}.
\]
for any arbitrary $\delta_1>0$ and therefore, we get a negligible contribution if $B\ll X^{1-\delta}$. Hence, we can also assume that for any $\delta>0$, $B\gg X^{1-\delta}$. By using the relation $A^{d-1}B=X^d$, we can furthermore assume, that for any $\delta>0$, $A\ll X^{1+\delta}$. By redefining $\delta$, we conclude the proof of Lemma \ref{lem:lem1}.
\end{proof}
\section{Analysis in $\setQ(\theta)$}
In the previous section, we have shown that the proof of Theorems \ref{thm:thm1} and \ref{thm:thm2} can be concluded if we find a suitable upper bound for the number of solutions $(n,a,b)$ of the Diophantine equation $a^{d-1}b=f(n)$. The idea is now to analyze this equation in the field $K$.
This idea is basically from Heath-Brown \cite{RHBsqfreen21}, where he derives an asymptotic formula for the number of $n\leq X$ such that $f(n)=n^2+1$ is square-free. Heath-Brown considers the corresponding equation $e^2f=n^2+1$ in the Gaussian integers $\setZ[i]$. Using that $\setZ[i]$ has unique factorization, he deduces that there are $\alpha,\beta\in\setZ[i]$ such that $n+i=\alpha^2\beta$ with $N(\alpha)=e$ and $N(\beta)=f$. Taking the imaginary part of this equation then gives a bi-homogeneous equation $G(x_0,x_1;y_1,y_2)=1$ which Heath-Brown then applies the approximate determinant method to.\smallskip\\
In our more general setting, there are three issues to tackle. First, in our problem it is not necessarily true that $\mathcal{O}_K=\setZ[\theta]$. Secondly, $\mathcal{O}_K$ might not have unique factorization and thirdly, our method will produce an equation system of $d-1$ bi-homogeneous auxiliary equations. We shall prove the following lemma:
\begin{lemma} \label{lem:lem2}
	Let $(n,a,b)$ be a triple counted by $\mathcal{N}(X;A,B)$. 
	We write $a=a_1a_2$ such that $a_1\divides\Delta^2(\theta)$ and
	$(a_2,\Delta^2(\theta))=1$. Then there exist one of at most 
	$O_{K,\mathcal{B}}(1)$ triples $(m,m_1,m_2)\in\setZ^3$ and $\alpha,
	\beta\in\setZ[\theta]$ such that 
	\begin{equation}\label{eq:eqmptheta}
		m(n+\theta)=\alpha^{d-1}\beta,
	\end{equation}
	with $|N(\alpha)|=m_1 a_2$ and $|N(\beta)|=m_2b$. Furthermore, the order
	of magnitude of the conjugates of $\alpha$ is given by
	\[
		|\alpha^\sigma|\asymp|N(\alpha)|^{1/d}.
	\]
\end{lemma}
\begin{proof}
The first step in our proof of Lemma \ref{lem:lem2} is to factorize the ideal
$J=\langle n+\theta\rangle$. Let $q$ be a prime divisor of $a_2$ and let $P$ be a prime ideal above $q$ such that $P\divides J$. Observe that 
\begin{equation}\label{eq:eqindex}
	[\mathcal{O}_K:\setZ[\theta]+P]\cdot [\setZ[\theta]+P:P]
	=[\mathcal{O}_K:P]=N(P).
\end{equation}
Hence, $[\mathcal{O}_K:\setZ[\theta]+P]\divides N(P)$, and thus, $[\mathcal{O}_K:\setZ[\theta]+P]$ is a power of $q$.
However, we also have
\[
	\Delta_K^2\cdot [\mathcal{O}_K:\setZ[\theta]+P]\cdot
	 [\setZ[\theta]+P:\setZ[\theta]]
	=\Delta_K^2\cdot[\mathcal{O}_K:\setZ[\theta]]=\Delta^2(\theta).
\]
Hence, $[\mathcal{O}_K:\setZ[\theta]+P]\divides\Delta^2(\theta)$ and
$q\nmid\Delta^2(\theta)$. Thus, $[\mathcal{O}_K:\setZ[\theta]+P]=1$.
Hence, by \eqref{eq:eqindex}, $N(P)=[\setZ[\theta]+P:P]$. This together
with $P\divides\langle n+\theta\rangle$ and $P\divides \langle q\rangle$
gives that
a set of representatives for $(\setZ[\theta]+P)/P$ is given by $\{0,1,\ldots,q-1\}$, so that in fact $N(P)=q$. Now assume that there are two prime ideals $P_1, P_2$ above $q$, both dividing $J$.
We can factorize $f(x)$ into irreducible factors modulo $q$:
\[
	f(x)\equiv\left(\prod_i (x+n_i)\right)\prod_j f_j(x)\ \pmod{q},
\]
where $\prod_i (x+n_i)$ is the product of the linear factors of $f$. Recall that $q\nmid\Delta^2(\theta)$ and that $N(P_1)=N(P_2)=q$. Thus, Kummer's Theorem on factorizations of prime ideals tells us that without loss of generality, $P_i=\langle q,\theta+n_i\rangle$ for $i=1,2$. Since $P_1$ divides both $\langle \theta+n_1\rangle$ and $J=\langle\theta+n\rangle$, we must have that 
$n-n_1\in P_1$ and hence $q=N(P_1)\divides N(n-n_1)$. Therefore, $n\equiv n_1\ \pmod{q}$ and similarly $n\equiv n_2\ \pmod{q}$. Thus, $n_1\equiv n_2\ \pmod{q}$ and therefore, $P_1=P_2$. Thus, the factor $P$ of $J$ occurs with multiplicity $d-1$. Hence, we get the ideal factorization
$J=I_1^{d-1}I_2$, with $N(I_1)=a_2$ and $N(I_2)=a_1^{d-1}b$.\bigskip\\
Let $\alpha_1\in I_1\setminus\{0\}$ be such that 
$|N(\alpha_1)|\leq c_K N(I_1)$
and let $\beta_1\in I_2\setminus\{0\}$ be such that $|N(\beta_1)|\leq c_K N(I_2)$. In particular, there exist non-zero ideals $J_1$ and $J_2$ of
$\mathcal{O}_K$ such that $\langle\alpha_1\rangle =I_1J_1$ and  $\langle\beta_1\rangle =I_2J_2$, with $N(J_1)\leq c_K$ and $N(J_2)\leq c_K$.
Recall that $\langle n+\theta\rangle=I_1^{d-1}I_2$ and hence 
$\langle n+\theta\rangle J_1^{d-1}J_2=\langle\alpha_1^{d-1}\beta_1\rangle$.
Thus,
The ideal $J_1^{d-1}J_2$ must be principal. We may therefore assume that
$J_1^{d-1}J_2=\langle\mu\rangle$ for some $\mu\in\mathcal{O}_K$ with
$|N(\mu)|\leq C_K^d$. Let $[\rho]_\sim$ denote the equivalence class of the equivalence relation defined on $\mathcal{O}_K$ by 
\[
	\rho_1\sim\rho_2\Leftrightarrow\rho_1,\rho_2\text{ are associates}.
\]
We then define $\lambda_K=\prod\rho$,
where the product is over the equivalence classes $[\rho]_\sim$ with $N(\rho)\leq C_K^d$. Thus, $\lambda_K$ is well defined up to a multiple of a unit and it depends only on $K$. Furthermore, $N(\mu)\divides N(\lambda_K)$. Thus, there are only $O_K(1)$ choices for $N(\mu)$, $N(J_1)$ and $N(J_2)$.
There are also only $O_K(1)$ choices for $a_1$.
We can conclude from the above that
\begin{equation}\label{eq:eqntheta}
	(n+\theta)\mu=\alpha_1^{d-1}\beta_1\epsilon_1,
\end{equation}
where $\epsilon_1\in\mathcal{O}_K^\times$. We shall now employ the following sub-lemma:
\begin{lemma}\label{lem:lemunit} Let $\gamma\in K\setminus\{0\}$. Then, there exists an $\epsilon\in\mathcal{O}_K^\times$ such that 
$|(\epsilon\gamma)^\sigma|\asymp_K |N(\gamma)|^{1/d}$.
\end{lemma}
\begin{proof}
Let $\underline v$ be the vector $(1,1,\ldots,1)\in\setZ^{r+s}$.
We write the embeddings of $K$ as $\sigma_1,\ldots,\sigma_r,
\sigma_{r+1},\ldots,\sigma_{r+2s}$ so that 
$\sigma_1,\ldots,\sigma_r$ are the real embeddings and $\sigma_{r+1},
\ldots,\sigma_{r+s}$ are each one of the complex conjugate pairs of embeddings. In particular, $r+2s=[K:\setQ]=d$.
We now define the maps $\theta$ and $\phi$ by
\begin{align*}
	\theta:K^\times &\rightarrow\setR^{r+s},\\
	\alpha &\mapsto(\log\vert\alpha^{\sigma_i}\vert)_{i=1,\ldots,r+s},
\end{align*}
and
\begin{align*}
	\phi:\setR^{r+s} &\rightarrow\setR^{r+s},\\
	(x_1,\ldots,x_{r+s})&\mapsto s\cdot\underline{v},
\end{align*}
where $s=d^{-1}(\sum_{k=1}^r x_k+2\sum_{k=r+1}^{r+s} x_k)$.
Note that $\theta$ is a homomorphism and that $\phi$ is a $\setR$-linear
idempotent map. Furthermore, 
$({d^{-1}\log\vert N(\gamma)\vert})\cdot\underline v=\phi(\theta(\gamma))$.
Thus, it suffices to show that there exists a unit $\epsilon$ such that
\begin{equation}\label{eq:equnitineq}
	\phi(\theta(\gamma))-\theta(\epsilon\gamma)\ll_K 1.
\end{equation}
We observe that by Dirichlet's Unit Theorem, $\theta$ maps the units onto
a lattice $\Lambda$ of dimension $r+s-1$. Furthermore,
$\phi(\theta(\gamma))-\theta(\gamma)\in\text{Ker}\ \phi$ and
for any unit $\epsilon$, $\theta(\epsilon)\in\text{Ker}\ \phi$. 
Thus, $\text{Ker}\ \phi$ is a $r+s-1$-dimensional vector space containing
the lattice $\Lambda$. Thus, by considering $\phi(\theta(\gamma))-\theta(\gamma)$ modulo the fundamental domain
of $\Lambda$, we can indeed pick $\epsilon$ such that \eqref{eq:equnitineq}
holds and the implied constant in \eqref{eq:equnitineq} depends only on the
size of the fundamental domain of $\Lambda$ which is determined by $K$.
This proves Lemma \ref{lem:lemunit}.
\end{proof}
Thus, by multiplying $\epsilon_1$ and $\alpha_1$ in \eqref{eq:eqntheta} with suitable units, if necessary, we may assume without loss of generality that 
$|\alpha_1^\sigma|\asymp|N(\alpha_1)|^{1/d}$ for all $\sigma$.
Recall that $\mathcal{B}=\{b_0,\ldots,b_{d-1}\}$ is an integral basis for $K$ and that $\{1,\theta,\ldots,\theta^{d-1}\}$ is a $\setQ$-basis
for $K$. Thus, there exist $r_{ij}\in\setQ$ such that
$b_i=\sum_j r_{ij}\theta^j$. Let $r$ be the least common multiple
of the denominators of the $r_{ij}$. Then, $r$ is an integer determined by $K$ and $\mathcal{B}$ such that $rb_i\in\setZ[\theta]$ for all $i$.
Multiplying \eqref{eq:eqntheta} by $r^d N(\mu)\mu^{-1}\in\mathcal{O}_K$, 
we obtain
\[
	(n+\theta)N(\mu)r^d=(r\alpha_1)^{d-1}(r\beta_1\epsilon_1 N(\mu)\mu^{-1}).
\]
We shall put $\alpha=r\alpha_1$ and 
$\beta=r\beta_1\epsilon_1 N(\mu)\mu^{-1}$. Note that
$\alpha_1$ and $\beta_1\epsilon_1 N(\mu)\mu^{-1}$ are in $\mathcal{O}_K$ and
thus, $\alpha,\beta\in\setZ[\theta]$. Also, 
$|\alpha^\sigma|\asymp |N(\alpha)|^{1/d}$.
Furthermore, $N(\alpha)=r^dN(J_1)a_2$ and 
$N(\beta)=r^dN(J_2)N(\mu)^{d-1}a_1^{d-1}b$.
Thus, if we set $m=N(\mu)r^d$, $m_1=r^dN(J_1)$ and
$m_2=r^dN(J_2)N(\mu)^{d-1}a_1^{d-1}$,
then $m, m_1$ and $m_2$ are determined by up to $O_{K,\mathcal{B}}(1)$ choices and the proof of Lemma \ref{lem:lem2} follows.
\end{proof}
\section{The Approximate Determinant Method}
Recall that
$\mathcal{B_\theta}=\{1,\theta,\ldots,\theta^{d-1}\}$ is 
the $\theta$-power basis of $K$ over $\setQ$.
We shall need the following lemma:
\begin{lemma} \label{lem:lemproj}
	There exist constants $c_0,\ldots,c_{d-1}\in K$ only
	depending on $K$ with the following property:\\
	If $\gamma=\sum_{i=0}^{d-1}r_i\theta^i$ with $r_i\in\setQ$ is
	an arbitrary element of $K$, then
	$r_i=\Tr(c_i\gamma)$.
\end{lemma}
\begin{proof}
We shall write $\mathbf{v}$ for a vector $(v_i)_{i=0,\ldots,d-1}$, and we also define the trace of a vector to be the coordinate-wise application of the trace, i.e. $Tr(\mathbf{v})=(Tr(v_i))_i$.
Furthermore, let $\mathcal{C}$ be the matrix $(Tr(\theta^{i+j}))_{j,i=0,\ldots,d-1}$. Then $\mathcal{C}\in\mathbb{M}_d(\setZ)$.
Observe that $\mathcal{C}$ has determinant $\det(\mathcal{C})=\Delta^2(\theta)\neq 0$
and hence is invertible with $\mathcal{C}^{-1}\in\mathbb{M}_d(\setQ)$.
Let $\mathbf{b}_\theta=(1,\theta,\ldots,\theta^{d-1})$. It is clear that $\Tr(\gamma\mathbf{b}_\theta)=\mathcal{C}\mathbf{r}$, and thus
\[
	\mathbf{r}=
	\mathcal{C}^{-1}\Tr(\gamma\mathbf{b}_\theta)
	=\Tr(\mathcal{C}^{-1}\mathbf{b}_\theta\gamma).
\]
Hence, if we define $\mathbf{c}=\mathcal{C}^{-1}\mathbf{b}_\theta$, then the claim of the lemma follows.
\end{proof}
We define the map
\[
	\pi_j(\gamma)=\Tr(c_j\gamma),
\]
which thus is the projection of $\gamma\in K$ to its $j$-th coordinate with respect to the basis $\mathcal{B}_\theta$.
Now let us go back to the equation \eqref{eq:eqmptheta}. For the remainder of this work, we shall write
\[
	\alpha=\sum_{i=0}^{d-1}x_i\theta^i,\qquad 
	\beta= \sum_{i=0}^{d-1}y_i\theta^i,
\]
with $x_i,y_i\in\setZ$. By Lemma \ref{lem:lemproj}, we have that
$y_i=Tr(c_i\beta)=\sum_\sigma c_i^\sigma\beta^\sigma$. Now we use the equation 
\eqref{eq:eqmptheta} to deduce that
\[
	y_i=\sum_\sigma\frac{m p c_i^\sigma}{\alpha^{(d-1)\sigma}}
	+O_K\left(\max_\sigma|\alpha^{-(d-1)\sigma}|\right).
\]
By Lemma \ref{lem:lem2}, $|\alpha^\sigma|=|N(\alpha)|^{1/d}\asymp A^{1/d}$
for all $\sigma$ so that
\[
	y_i=y_{i,0}+O(A^{-(d-1)/d}),
\]
say. To simplify our notation, we shall now assume that $|x_i|\leq |x_0|$  and $|y_i|\leq |y_0|$ for all $i$. This assumption will in fact be without loss of generality and $x_0$ and $y_0$ could be replaced with any largest
$x_k$ and $y_l$, say. The fact that $N(\alpha)\asymp A$
and $N(\beta)\asymp B$ implies 
\begin{align*}
	&|x_i|\ll A^{1/d}, |y_i|\ll B^{1/d}\text{ for all } i,\\
	 \text{and}\quad &|x_0|\gg A^{1/d}, |y_0|\gg B^{1/d}.
\end{align*}
Hence, we also have $|y_{0,0}|\gg B^{1/d}$. Thus, if we set $s_i=x_i/x_0$
and $t_i=y_i/y_0$ for $i=0,\ldots,d-1$ then
\[
	t_j=\frac{y_{j,0}}{y_{0,0}}+O\left(\frac{1}{X}\right)
\]
for $j=1,\ldots,d-1$.
Note that ${y_{j,0}}/{y_{0,0}}$ is a rational function of degree $(d-1)^2$ in the variables $s_1,\ldots,s_{d-1}$ with coefficients in $\setC$. Thus, 
we set
\begin{equation}\label{eq:eqtj}
	t_j=F_j(s_1,\ldots,s_{d-1})+O\left(\frac{1}{X}\right),
\end{equation}
where $F_j={y_{j,0}}/{y_{0,0}}$. Note that $|y_{0,0}|\gg B^{1/d}$ and hence,
the denominator of $F_j$ is non-zero provided $X$ is large enough.
We now want to find an upper bound on the
number of points $(x_0,\ldots,x_{d-1})\in\setZ^d$. 
We continue by splitting the possible range of the $(d-1)$-tuple
$(s_1,\ldots,s_{d-1})$ into $O(M^{d-1})$ boxes of the shape
\begin{equation}\label{eq:eqbox}
	\hat B=\{(s_1,\ldots,s_{d-1})\in\setQ^{d-1}: 
	s_i\in (s_{i,0},s_{i,0}+O(M^{-1})),\ i=1,\ldots,d-1\}.
\end{equation}
Our goal is to find an upper bound on the number of points $(s_1,\ldots,s_{d-1})\in\setQ^{d-1}$ inside one such box.
We impose the condition $M^{d-1}\ll \min(A,X)$ on $M$.
Thus, we may now fix all $s_{i,0}\ll 1$ and one such box $\hat B$. Hence,
we may assume that 
\begin{equation}\label{eq:eqsifinal}
	s_i=s_{i,0}+\alpha_i,\qquad \alpha_i\ll M^{-1},
\end{equation}
for $i=1,\ldots,d-1$. Next,
consider equation \eqref{eq:eqtj}. Since $M\ll X^\frac{1}{d-1}$, and since $F_j$ has
no zeros in the denominator, we can assume that $F_j$ has partial derivatives
of all orders, and thus, we may apply Taylor's Theorem to deduce that
\begin{equation}\label{eq:eqtjfinal}
	t_j=P_j(\alpha_1,\ldots,\alpha_{d-1})+\beta_j,\qquad
	\beta_j\ll X^{-1},
\end{equation}
for $j=1,\ldots,d-1$.
Here, $P_j$ is a polynomial in $d-1$ variables of sufficiently large degree with coefficients of size $O(1)$. This is because the coefficients only depend on
$s_{1,0},\ldots,s_{d-1,0}$, and $s_{i,0}\ll 1$ for all $i$, and because the denominator of $F_j$ is $\gg 1$.\bigskip\\
We are now ready to apply the approximate determinant method.
The idea is to consider the monomials $s_it_j$, for $(i,j=0,\ldots,d-1)$,
where we recall that $s_0=t_0=1$ by definition.
We write these monomials as $m_r(\mathbf{s},\mathbf{t})$, 
where $r\leq R$, with $R=d^2$. Assume that the solutions of \eqref{eq:eqmptheta} with $\mathbf{s}\in \hat B$ are
$(\mathbf{s}^{(1)},\mathbf{t}^{(1)}),\ldots,
(\mathbf{s}^{(J)},\mathbf{t}^{(J)})$.
Then we define the $J\times R$ matrix 
$\mathcal{M}$ with $(j,r)$-th entry being 
$m_r(\mathbf{s}^{(j)},\mathbf{t}^{(j)})$. Our aim is to show that $\mathcal{M}$ has rank strictly less than $R$, provided we chose $M$ appropriately. This will then enable us to show that there is a non-zero vector $\mathbf{v}$ such that $\mathcal{M}\mathbf{v}=0$. Thus, if we define the the polynomial
$C_{\hat B}(\mathbf{s},\mathbf{t})=\sum_{r=1}^R v_r 
m_r(\mathbf{s},\mathbf{t})$, then
$C_{\hat B}(\mathbf{s}^{(j)},\mathbf{t}^{(j)})=0$ for all our solutions
$(\mathbf{s}^{(j)},\mathbf{t}^{(j)})$ with $\mathbf{s}^{(j)}\in\hat B$. 
Observe that $\mathcal{M}$ is a matrix with rational entries
and the vector $\mathbf{v}$
can be constructed from subdeterminants of $\mathcal{M}$. Thus,
$\mathbf{v}\in\setQ^R$ and by clearing the denominators of
the coefficients of $C_{\hat B}$, we may assume that $C_{\hat B}$ has integer coefficients
of size bounded by a power of $X$. Thus, we shall assume that $\Vert C_{\hat B}\Vert\ll X^\kappa$, say.\smallskip\\
We now proceed to show that $\mathcal{M}$ has rank strictly less
than $R$. Without loss of generality $J\geq R$, since otherwise this is trivial. Thus, it suffices to show that every $R\times R$ subdeterminant of 
$\mathcal{M}$ vanishes. Without loss of generality, let us consider the subdeterminant $\Delta$ of $\mathcal{M}$ coming from the first 
$R$ rows and columns. Note that $j$-th row of $\mathcal{M}$ has entries with common denominator of size $A^{1/d}B^{1/d}$. Hence, if we can show that $\Delta\ll (AB)^{-\frac{R}{d}}$, then $\Delta=0$.  Substituting \eqref{eq:eqsifinal}
and \eqref{eq:eqtjfinal} into our matrix, we obtain a generalized
$R\times R$ Vandermonde determinant in $\alpha_i$ and
$\beta_j$ with entries of the shape
\[
	(s_{i,0}+\alpha_i)^{\epsilon_1}
	(P_j(\alpha_1,\ldots,\alpha_{d-1})+\beta_j)^{\epsilon_2},
\]
with $i,j\in\{1,\ldots,d-1\}$ and $\epsilon_1,\epsilon_2\in\{0,1\}$.
Note that we have
\[
	\alpha_i\ll T_1^{-1},\qquad\beta_j\ll T_2^{-1},
\]
where $T_1\asymp M$ and $T_2\asymp X$.
We proceed to find an upper bound for $\Delta$. We order the monomials
$T_1^{-a}T_2^{-b}$ in decreasing size, $1=M_0\geq M_1\geq\ldots$, say. By  Lemma 3 of Heath-Brown \cite{RHBdiffthreekpow},
\[
	\Delta\ll_d\prod_{r=1}^R M_r.
\]
We have that $T_1\ll T_2^{1/(d-1)}$ and $d\geq 3$. We first consider the 
case $d\geq 4$ in which the largest monomial sizes are 
$1\geq T_1^{-1}\geq T_1^{-2}\geq T_1^{-3}$.
The number of monomials of degree $D$ in $N$ variables is
\[
	n(D,N)=\binom{N+D-1}{D}.
\]
The number of monomials $M_h$ with size $T_1^{-m}$ is
$n(m,d-1)$. Observe that
\[
	\mathcal{S}:=\sum_{m=0}^2 n(m,d-1)=\frac{1}{2}d(d+1)<d^2=R,
\]
and
\[
	\sum_{m=0}^3 n(m,d-1)=\frac{1}{6}d(d+1)(d+2)\geq d^2=R
\]
for all $d\geq 3$. Thus, the factors in $\prod_r M_r$ are all of size
$T_1^{-m}$, with $m\in\{0,1,2,3\}$, and we obtain the bound
$\Delta\ll T_1^{-\xi}\ll M^{-\xi}$, where
\[
	\xi=\sum_{m=0}^2 m n(m,d-1) + 3\cdot (R-\mathcal{S})
	=\frac{1}{2}(d-1)(5d+2).
\]
Thus, if $(AB)^dM^{-\xi}\ll 1$ 
with a suitably small implied constant, then
$\Delta=0$.\smallskip\\
We apply a similar argument in the case $d=3$. We will pick $M$ such that $X^{1/3}\leq M\leq X^{1/2}$. As above, we consider an $R\times R$
matrix with $R=9$. The largest 9 monomials are
$1,\alpha_1,\alpha_2,\alpha_1^2,\alpha_1\alpha_2,\alpha_2^2,\beta_1,\beta_2,
\alpha_1^3$. Thus, $\Delta\ll M^{-11}X^{-2}$. Therefore, 
$M^{11}\gg (AB)^3X^{-2}$ implies $\Delta=0$. And indeed, 
$(AB)^{3/11}X^{-2/11}\gg X^{1/3}$, since $A,B\gg X^{1-\delta}$, 
if $\delta$ is small enough. Thus, we have proved the following lemma:
\begin{lemma}\label{lem:auxeq}
	Let $\epsilon>0$. Assume that $M$ is an integer satisfying 
	\[
		(AB)^{\frac{2d}{(5d+2)(d-1)}}\ll M\ll \min(X^{1/{(d-1)}},
			A^{1/{(d-1)}}).
	\]
	if $d\geq 4$ and
	\[
		(AB)^{3/11}X^{-2/11}\ll M\ll \min(X^{1/2},A^{1/2})
	\]
	if $d=3$.
	Then for any box $\hat B$ of shape \eqref{eq:eqbox}, there exists
	a non-zero bilinear integer form 
	$C_{\hat B}(\mathbf{x};\mathbf{y})$,
	with coefficients of size $\Vert C_{\hat B}\Vert\ll X^\kappa$ such that
	$C_{\hat B}(\mathbf{x},\mathbf{y})=0$ for all solutions of 
	\eqref{eq:eqmptheta} with $\mathbf{s}\in\hat B$.
\end{lemma}
\section{Counting Points inside a box}
Observe that by Lemma \ref{lem:lem2},
\[
	\sum_{j=0}^{d-1}y_j\pi_i(\alpha^{d-1}\theta^j)=\pi_i(\alpha^{d-1}\beta)
	=m\pi_i(n+\theta)
	=\begin{cases}
		m & \text{ if } i=1\\
		0 & \text{ if } 2\leq i\leq d-1
	\end{cases}		
\]
for $1\leq i\leq d-1$. 
We ignore the equation for $i=0$ because we want the left-hand side of
our equations to have size $O_K(1)$.
We define $G_{i,j}(\mathbf{x})=\pi_i(\alpha^{d-1}\theta^j)$ for
$1\leq i\leq d-1$ and $0\leq j\leq d-1$.
Then, the $G_{i,j}(\mathbf{x})$ are forms of degree $d-1$ with integer coefficients in the variables $\mathbf{x}$. We shall write the form $C_{\hat B}$ from Lemma \ref{lem:auxeq} as $C_{\hat B}(\mathbf{x};\mathbf{y})=\sum_{j=0}^{d-1} y_j G_{d,j}(\mathbf{x})$. Then the forms $G_{d,j}(\mathbf{x})$ are linear with integer coefficients. Let $\mathbf{e}=(m,0,0,\ldots,0)\in\setZ^d$ and
define the matrix
\[
	\mathcal{G}=\mathcal{G}(\mathbf{x})
	=(G_{i,j}(\mathbf{x}))_{i=1,\ldots,d;\ j=0,\ldots,d-1}.
\]
Then $\mathcal{G}$ is a $d\times d$ matrix and we have the equation system 
\begin{equation}\label{eq:sys1}
	\mathcal{G}(\mathbf{x}).\mathbf{y}=\mathbf{e}.
\end{equation}
Our next aim is to apply a change of variables to the vector $\mathbf{x}$ in order to rewrite the condition $\mathbf{s}\in \hat B$ more conveniently. We will proceed similarly to Heath-Brown \cite{RHBsqfreen21}.
We recall that for $i=1,\ldots,d-1$ we have that
$s_i=s_{i,0}+O(1/M)$ and that $s_{i,0}=x_{i}/x_{0}$, where
$x_{0}\asymp A^{1/d}$. Thus,
\[
	|x_i-s_{i,0}x_{0}|\ll A^{1/d}/M.
\]
Next, we define the linear operator
$T:\setR^d\rightarrow\setR$ by
\[
	T(x_0,x_1,\ldots,x_{d-1})=(MA^{-1/d}(x_1-s_{1,0}x_{0}),\ldots,
		MA^{-1/d}(x_{d-1}-s_{d-1,0}x_{0}),A^{-1/d}x_0).
\]
Then
\[
	\Lambda = \{T(\mathbf{x}):\mathbf{x}\in\setZ^d\}
\]
is a lattice of determinant $\det(\Lambda)=M^{d-1}A^{-1}$.
If we define the rectangle
\[
	R=\{(v_0,\ldots,v_{d-1})\in\setR^d:|v_i|\ll 1\text{ for }
	 0\leq i\leq d-1\},
\]
where the implied constants are suitably chosen, then we are interested in counting the points falling into $\Lambda\cap R$.
By considering a basis of shortest lattice vectors in $\Lambda$, $\mathbf{g}^{(0)},\ldots, \mathbf{g}^{(d-1)}$, say, we can change the basis so that the variables $x_0,\ldots,x_{d-1}$ become $u_0,\ldots,u_{d-1}$. Furthermore, if we define $U_i$ to be a suitable constant times $|\mathbf{g}^{(i)}|^{-1}$ for $i=0,\ldots, d-1$ then $|u_i|\leq U_i\leq U_0$ for all $i$. Furthermore, the $U_i$ satisfy
\[
	\prod_{i=0}^{d-1} U_i \asymp \det(\Lambda)^{-1}=\frac{A}{M^{d-1}}.
\]
Thus, our equation system \eqref{eq:sys1} becomes
\[
	\mathcal{G}(L(\mathbf{u})).\mathbf{y}=\mathbf{e},
\]
with $|u_i|\leq U_i\leq U_0$ and $U_0^d\gg A/M^{d-1}\gg 1$, and
where $L$ is the invertible linear function such that $L(\mathbf{u})=\mathbf{x}$.\bigskip\\
Our goal is now to count the contribution from each of our boxes $\hat B$ using
the equation system \eqref{eq:sys1}. More precisely, we shall prove the following:
\begin{lemma}\label{lem:lem3}
	The number of $\mathbf{x}$ with $\mathbf{s}\in\hat B$ 
	which satisfy the equation system 
	\eqref{eq:sys1} is $O(U_0^{d-1}X^\epsilon)$.
\end{lemma}
\begin{proof}
Let $\Delta=\Delta(\mathbf{x})$
 be the determinant of $\mathcal{G}(\mathbf{x})$ and for $i=0,\ldots,d-1$,
let $\Delta_i=\Delta_i(\mathbf{x})$ be the determinant of the matrix
that we obtain when we replace the $(i+1)$-th column of $\mathcal{G}$ by the vector $\mathbf{e}$. Then,
by Cramer's Rule we obtain the equation system
\[
	\Delta y_i=\Delta_i\qquad (i=0,\ldots,d-1).
\]
We can see that $\Delta$ is a form of degree $(d-1)^2+1$ in $\mathbf{x}$,
and that $\Delta_i$ are forms of degree $(d-1)(d-2)+1$ in $\mathbf{x}$.\bigskip\\
We proceed to show that $\Delta$ does not vanish identically.
Consider the $d\times d$ matrix 
$\mathcal{G}_1=(\sigma_i(\theta^j))_{i=1,\ldots,d;j=0,\ldots,d-1}$.
Then
\[
	\det(\mathcal{G}_1)=\prod_{1\leq i<j\leq d}
    (\sigma_i(\theta)-\sigma_j(\theta))=\Delta(\theta)\neq 0,
\]
and thus, $\mathcal{G}_1$ is invertible and we may define the linear forms $F_1,\ldots,F_{d}$ by the following vector-matrix multiplication:
\[
	(F_1,\ldots,F_{d})=(G_{d,0},\ldots,G_{d,d-1})\mathcal{G}_1^{-1}.
\]
For the purposes of showing $\Delta\not\equiv 0$, the $F_i$ may be seen as linear forms in the variables $\mathbf{a}=(\sigma_1(\alpha),\ldots,\sigma_d(\alpha))$ with coefficients in $\bar\setQ$. Furthermore, let $\mathcal{G}_2$ be the diagonal $d\times d$ matrix with diagonal entries $\sigma_1(\alpha^{d-1}),\ldots,\sigma_d(\alpha^{d-1})$. Then, we may factorize the matrix $\mathcal{G}$ as follows:
\[
	\mathcal{G}=
	\begin{pmatrix}
	\pi_1(\alpha^{d-1}) & \ldots & \pi_1(\alpha^{d-1}\theta^{d-1})\\
	\vdots &  &\vdots \\
	\pi_{d-1}(\alpha^{d-1}) & \ldots & \pi_{d-1}(\alpha^{d-1}\theta^{d-1})\\
	G_{d,0} & \ldots & G_{d,d-1}
	\end{pmatrix}
	=
	\begin{pmatrix}
	\sigma_1(c_1) & \ldots & \sigma_d(c_1) \\
	\vdots & & \vdots \\
	\sigma_1(c_{d-1}) & \ldots & \sigma_d(c_{d-1}) \\
	\frac{F_1(\mathbf{a})}{\sigma_1(\alpha^{d-1})} 
	& \ldots & \frac{F_d(\mathbf{a})}{\sigma_d(\alpha^{d-1})}
	\end{pmatrix}\mathcal{G}_2\mathcal{G}_1.
\]
Thus,
\begin{equation}\label{eq:eqdetG}
	\det(\mathcal{G})=\pm N_{K/\setQ}(\alpha^{d-1})
	\Delta(\theta)\sum_{i=1}^d\frac{F_i(\mathbf{a})}{\sigma_i(\alpha^{d-1})}(-1)^{i+1}\det(\mathcal{A}_i),
\end{equation}
where 
\[
	\mathcal{A}=(\sigma_j(c_i))_{i=0,\ldots,d-1;j=1,\ldots,d}
\]
is a $d\times d$ matrix and for $r=1,\ldots,d$, the $(d-1)\times (d-1)$ matrix $\mathcal{A}_r$ is obtained by deleting the first row and the $r$-th column of $\mathcal{A}$. Recall from Lemma \ref{lem:lemproj} that the vector
$\mathbf{c}$ is given by $\mathcal{C}\mathbf{c}=\mathbf{b}_\theta$ and
that $\det(\mathcal{C})=\Delta^2(\theta)\in\setQ\setminus\{0\}$. For
$i=1,\ldots,d$, let $\mathcal{C}_i$ be the matrix obtained by replacing
the $i$-th column of $\mathcal{C}$ by $\mathbf{b}_\theta$.
Then, by Cramer's Rule, 
\[
	\det(\mathcal{C})c_i=\det(\mathcal{C}_{i+1})\qquad (i=0,\ldots,d-1).
\]
Therefore, we can see that 
\[
	\Tr(c_i)=(\Delta(\theta))^{-2}\Tr(\det(\mathcal{C}_{i+1}))=
	\begin{cases}
		1 & \text{if } i=0\\
		0 & \text{if } 1\leq i\leq d-1.
	\end{cases}
\]
And thus, $\det(\mathcal{A}_i)+\det(\mathcal{A}_{i+1})=0$ for all 
$i=1,\ldots,d-1$. In particular,
\[
	(-1)^{i+1}\det(\mathcal{A}_{i})=\det(\mathcal{A}_1)
\]
for all $i$. The matrix $\mathcal{C}$ has entries in $\setZ$ and hence
we get $\mathcal{C}\sigma(\mathbf{c})=\sigma(\mathbf{b}_\theta)$
for any embedding $\sigma$. Hence
$\mathcal{C}\mathcal{A}=\mathcal{G}_1^T$ and therefore 
$\det(\mathcal{A})=(\Delta(\theta))^{-1}$. By expanding $\det(\mathcal{A})$
along the first row, we get
\[
	(\Delta(\theta))^{-1}=
	\sum_{i=1}^d\sigma_i(c_0)(-1)^{i+1}\det(\mathcal{A}_i)
	=\det(\mathcal{A}_1)\Tr(c_0)=\det(\mathcal{A}_1).
\]
Putting this into \eqref{eq:eqdetG}, we obtain that
\[
	\det(\mathcal{G})=\pm N_{K/\setQ}(\alpha^{d-1})
	\sum_{i=1}^d\frac{F_i(\mathbf{a})}{\sigma_i(\alpha^{d-1})}
	=\pm \sum_{i=1}^d F_i(\mathbf{a})
	\prod_{\substack{j=1\\ j\neq i}}^d a_j^{d-1},
\]
where $a_j=\sigma_j(\alpha)$. The auxiliary form $C_{\hat B}$ created by Lemma \ref{lem:auxeq} does not vanish identically. Hence, at least one of the forms $F_i$ does not vanish identically. Thus, we may assume that $F_r$, say is not identically zero. Note that
\[
	\det(\mathcal{G})\equiv \pm F_r(\mathbf{a})
	\prod_{\substack{j=1\\ j\neq r}}^d a_j^{d-1}\quad \pmod {a_r^{d-1}}.
\]
By considering the right-hand side as a polynomial in $a_r$, it cannot vanish identically because $F_r(\mathbf{a})$ is a non-zero linear form. Therefore, $\det(\mathcal{G})$ does not vanish identically.\bigskip\\
Recall that the total degree of $\Delta=\det(\mathcal{G})$ is $D=(d-1)^2+1$, say. Our next aim is to apply a further linear change of variables, $\mathbf{u}=M\mathbf{v}$, so that $v_{0}^D$ occurs with non-zero
coefficient in $\Delta$. We have shown that $\Delta$ does not vanish identically and hence we may use Theorem 1 of Heath-Brown \cite{cime}
to deduce that there exists a primitive vector $\mathbf{u_0}\in\setZ^d$
such that $\Delta(\mathbf{u_0})\neq 0$ and $|\mathbf{u_0}|\ll_d 1$. 
Let $\mathbf{u_0}$ be the first column of $M$. It is then possible to
complete $M$ to a $d\times d$ matrix having integer entries and determinant $1$. The entries of $M$ are all bounded in size in terms of $\mathbf{u_0}$ so that $\max_{i,j} |M_{i,j}|\ll_d 1$. Furthermore, $\Delta(M.(1,0,\ldots,0)^T)=\Delta(\mathbf{u_0})\neq 0$. Since $M$ is unimodular, $M^{-1}$ has entries in $\setZ$. This ensures indeed that $v_i\in\setZ$ with $v_i\ll U_0$ for all $i$. Thus, the linear change of variables $\mathbf{u}=M\mathbf{v}$ ensures that $v_0^D$ occurs with non-zero coefficient in $\Delta$ but it does not make any difference to our argument otherwise. We shall therefore consider $\Delta'=\Delta(M\mathbf{v})$
and $\Delta'_0=\Delta_0(M\mathbf{v})$ in the remaining argument.\bigskip\\
The number of vectors $\mathbf{v}$ with $\Delta'=0$ is $O_d(U_0^{d-1})$. This can be established after factorizing $\Delta'$ into
irreducible factors and by applying \cite[Theorem 1]{cime} to each factor. Thus, we may now consider the points $\mathbf{v}$
with $\Delta'(\mathbf{v})\neq 0$.\bigskip\\
There are $O(U_0^{d-1})$
choices for the points $(v_1,\ldots,v_{d-1})$. We fix one such point and
we will show that there are only $O_\epsilon(X^\epsilon)$ choices for $v_0$.  Consider $\Delta'=\Delta'(v_0)$ and $\Delta'_0=\Delta'_0(v_0)$ as polynomials in $v_0$ with integer coefficients. By dividing the polynomial equation $\Delta' y_0=\Delta'_0$ by a common factor, if necessary, we may assume that
$F(v_0)y_0=F_0(v_0)$, where $F(x)$ and $F_0(x)$ are coprime polynomials in $\setZ[x]$. Also note that the degree of $F$ must be at least 1 because the degree of $\Delta'$ in $v_0$ is strictly larger than the degree of $\Delta'_0$ in $v_0$. By the Euclidean algorithm, there exist polynomials $g_0(x),g(x)\in\setZ[x]$ such that 
\begin{equation}\label{eq:eqgg0}
	g_0(x)F_0(x)+g(x)F(x)=R,
\end{equation}
where $R$ is the resultant of $F$ and $F_0$ in the variable $v_0$, depending on $v_1,\ldots,v_{d-1}$. Note that $R\neq 0$ because $F$ and $F_0$ do not have a common root. The resultant can be defined as the determinant of the Sylvester Matrix of $F$ and $F_0$ which only depends on the coefficients of $F$ and $F_0$. Recall that these coefficients only depend on the polynomial $f$ and on the coefficients of the auxiliary equation created by Lemma \ref{lem:auxeq}. Thus, $R$ is bounded by a power of $X$.
Substituting $F(v_0)y_0=F_0(v_0)$ into \eqref{eq:eqgg0},
we get that $F(v_0)\divides R$, and hence there are only $O_\epsilon(X^\epsilon)$ choices for $v_0$.
And therefore, if we put all cases together, then there are $O(U_0^{d-1}X^\epsilon)$ choices for $(v_0,\ldots,v_{d-1})$ and each choice of $\mathbf{v}$ gives exactly one choice for $\mathbf{x}$ with $\mathbf{s}\in \hat B$. This finishes the proof of Lemma \ref{lem:lem3}.
\end{proof}
\section{Finishing the Proof}
We shall now finish the proof of Theorem \ref{thm:thm1} and Theorem \ref{thm:thm2}. In the previous section we have shown that the number of $\mathbf{x}$ with $\mathbf{s}\in\hat B$ that satisfy \eqref{eq:sys1} is $O(U_0^{d-1}X^\epsilon)$. 
It thus remains to sum up the contributions from all boxes.
We write the shortest non-zero lattice vector $\mathbf{g}^{(0)}$ in $\Lambda$ from the previous section as
\[
	\mathbf{g}^{(0)}=(MA^{-1/d}(x_1-s_{1,0}x_{0}),\ldots,
		MA^{-1/d}(x_{d-1}-s_{d-1,0}x_{0}),A^{-1/d}x_0).
\]
Recall further that $U_0$ was defined to be a suitable constant times
$|\mathbf{g}^{(0)}|^{-1}$. Thus,
\begin{align*}
	&U_0|x_i-s_{i,0}x_0|\ll A^{1/d}M^{-1} \qquad (i=1,\ldots,d-1),\\
	&U_0|x_0|\ll A^{1/d}.
\end{align*}
We produce the boxes $\hat B$ by taking $s_{i,0}=\frac{z_i}{M}$, where $z_i$ is an integer of exact order $M$. Hence, the number of boxes $\hat B$ for which $U_0\sim U$ is at most the number of $(2d-1)$-tuples
$(x_0,\ldots,x_{d-1},z_1,\ldots,z_{d-1})$ for which
\[
	z_ix_0 = Mx_i + O\left(\frac{A^{1/d}}{U}\right),\qquad
	x_0\ll\frac{A^{1/d}}{U},\qquad z_i\asymp M,
\]
where $i=1,\ldots,d-1$. Recall that by asumption on $M$, we have that
$U^d\gg A/M^{d-1}\gg 1$. For $i=1,\ldots,d-1$, observe that
\[
	x_i\ll x_0 +O(M^{-1/d}).
\]
Hence, if $x_0=0$ then $\mathbf{g}^{(0)}=\mathbf{0}$ which is impossible. 
So $|x_0|\geq 1$ and therefore $A^{1/d}U^{-1}\gg 1$. We may therefore assume that $x_0\neq 0$ and $x_i\ll x_0$ for all $i$. There are
$O(A^{1/d}U^{-1})$ choices for $x_0$.
Fix one such $x_0$.
For $i=1,\ldots,d-1$, there exist integers $r_i\ll A^{1/d}U^{-1}$ such that
\begin{equation}\label{eq:linDio}
	z_ix_0=Mx_i+r_i.
\end{equation}
Thus, there are $O\left((A^{1/d}U^{-1})^{d-1}\right)$
choices for $r_1,\ldots,r_{d-1}$, since $A^{1/d}{U}^{-1}\gg 1$. The congruences
\[
	Mx_i\equiv -r_i\ \pmod{x_0}
\]
determine $x_1,\ldots,x_{d-1}$ modulo $x_0$, provided $x_0$ and $M$ are coprime.
Indeed,
\[
	1\leq |x_0|\ll\frac{A^{1/d}}{U}\ll M^{1-1/d}<M,
\]
if $X$ and hence $M$ are large enough. Furthermore, we may pick $M$ to be a prime in Lemma \ref{lem:auxeq}, which then indeed insures that $M$ and
$x_0$ are coprime. Now, $x_i$, $x_0$ and $r_i$ determine $z_i$ by \eqref{eq:linDio}.
Thus, the number of intervals for which
$U_0\sim U$ is  $O\left(\frac{A}{U^d}\right)$.
By Lemma \ref{lem:lem3}, each of these intervals contributes at most $O(U^{d-1}X^\epsilon)$ choices
for $\mathbf{x}$. Thus, altogether there are 
\[
	\ll \frac{A}{U^d}\cdot  U^{d-1} X^{\epsilon}
	= \frac{A}{U}X^{\epsilon}\ll X^{\epsilon}(AM)^{1-1/d}
\] 
choices for $\mathbf{x}$. Each such $\mathbf{x}$ determines $\alpha$ and hence, each $\mathbf{x}$ gives $O(1)$ choices for $a_2$ by Lemma \ref{lem:lem2}. Note that there are only $O_K(1)$ values for $a_1$ such that
$a_1\divides\Delta^2_K$. Hence, 
$\#\{a:(n,a,b) \text{ is counted by }\mathcal{N}(X;A,B)\}\ll X^{\epsilon}(AM)^{1-1/d}$.
Thus, by a trivial estimate, we deduce that for each $\epsilon,\epsilon'>0$,
\begin{align*}
	\mathcal{N}(X;A,B)&\ll\sum_{a}
	\#\{n\sim X: f(n)\equiv 0\ \pmod{a^{d-1}}\}
	\ll\sum_a\rho(a^{d-1})\left(\frac{X}{a^{d-1}}+1\right)\\
	&\ll X^{\epsilon}\sum_a 1\ll X^{2\epsilon}(AM)^{1-1/d}
	\ll M^{1-1/d}X^{1-1/d+\epsilon'}.
\end{align*}
Thus, it suffices if there exists a $\eta>0$ such that $M^{d-1}\ll X^{1-\eta}$.
By Lemma \ref{lem:auxeq}, we can indeed pick such an $M$ because
$\frac{4d}{5d+2}<\frac{4}{5}$ for $d\geq 4$. If $d=3$ then $M$ is essentially $X^{4/11}$ and $8/11<1$. This finishes the proof of Theorem \ref{thm:thm1} and \ref{thm:thm2}.


\end{document}